\documentclass[12pt]{amsart}
\usepackage{graphicx}
\usepackage{subcaption} 
\usepackage{xfrac}
\usepackage{faktor}
\usepackage{tikz-cd}
\usepackage{cite}
\usepackage{mathrsfs}
\usepackage{hyperref}
\usepackage{amsmath}
\usepackage{amssymb}
\hypersetup{bookmarks=true,
unicode=true,
colorlinks=true,
citecolor=black,
linkcolor=black,
urlcolor=black,
plainpages=false,
pdfpagelabels=true}
\usepackage{url}	
\allowdisplaybreaks 
\usepackage{tikz-cd}
\usepackage{pgf}
\usepackage{xcolor}

\usepackage{comment} 
\usepackage[all]{xy}
\usetikzlibrary{arrows.meta,calc,positioning}
\newtheorem{teo}{Theorem}[section]

\newtheorem{cor}[teo]{Corollary}

\newtheorem{lemma}[teo]{Lemma}
\newtheorem{prop}[teo]{Proposition}

\theoremstyle{definition}
\newtheorem{definition}[teo]{Definition}

\newtheorem{example}[teo]{Example}

\theoremstyle{remark}
\newtheorem{remark}[teo]{Remark}

\numberwithin{figure}{section}

\newcommand{\CC}{\mathcal{C}}
\newcommand{\DD}{\mathcal{D}}

\newcommand{\cat}{\mathop{\mathrm{cat}}}

\newcommand{\co}{\colon}

\newcommand{\Fc}{\mathcal{F}}
\newcommand{\Dc}{\mathcal{D}}
\newcommand{\Si}{\mathcal{S}}

\newcommand{\I}{\mathcal{I}}

\newcommand{\id}{\mathrm{id}}
\newcommand{\N}{\mathrm{N}}

\newcommand{\op}{\mathrm{op}}

\newcommand{\Top}{\mathrm{Top}}
\newcommand{\Uc}{\mathcal{U}}

\newcommand{\Hom}{\mathrm{Hom}}

\newcommand{\Obj}{\mathrm{Obj}}
\newcommand{\Set}{\mathrm{Set}}

\begin{document}
\title[From old categorical models to new ones]{From old categorical models to new ones through open covers}
\thanks{}
\author[Isaac Carcacía Campos]{%
Isaac Carcacía Campos
}
\address{%
Isaac Carcacía Campos
\\
Departamento de Matem\'aticas, Universidade de Santiago de Compostela, 15782-SPAIN}
\email{isaac.c.campos@usc.es}
\begin{abstract}
The nerve theorem provides a combinatorial model for the homotopy type of a topological space from a suitable open cover. We extend this local-to-global approach by replacing the intersections of the cover with compatible categorical models. These local models are assembled through the Grothendieck construction, yielding a small category whose classifying space has the weak homotopy type of the original space. Under suitable hypotheses, the resulting category is finite and acyclic.

We develop two versions of this construction. The first is indexed by the usual poset of nonempty finite intersections of the cover. The second uses the membership poset, inspired by the finite-space construction of Sancho de Salas, which retains only the membership patterns realized by points of the space and is defined for point-finite covers. The classical nerve theorem and its componentwise variant are recovered as particular cases. The construction is recursive: categorical models of local pieces, together with functors representing their inclusions, can be combined to produce categorical models of more complex spaces.
\end{abstract}
\keywords{Homotopy Colimit, Nerve Theorem, Open Cover, Grothendieck construction, Aclycic category.}
\subjclass[2020]{
Primary 55U35, 55U40;
Secondary 55U10, 55P10, 18B35, 55P40.
}
\maketitle
\section*{Introduction}
One of the most important problems in computational topology is how to feed a computer a usable description of a topological space. The usual definition of a topological space is too general for direct computation, since its topology will usually involve an infinite set and an infinite collection  of open subsets. For this reason, one usually seeks a finite discrete or combinatorial model from which relevant topological invariants can be computed.

From the perspective of algebraic topology, a model need not recover the space point by point: it should preserve the homotopical or homological information of interest. Simplicial complexes, simplicial sets, posets and small categories provide different forms of discrete data for this purpose \cite{Quillen1973,McCord1966,Kozlov2008,BARMAK,Raptis2010,Thomason1980}. The difficulty is to construct a manageable model for a given space, rather than merely to know that one exists.

Open covers offer a classical way to do so. The nerve of a cover records which finite families of its members have nonempty intersection. If every nonempty finite intersection is weakly contractible, the realization of the nerve recovers the weak homotopy type of the space \cite{McCord1967,BauerKerberRollRolle2023,Ramras2026}. Suitable covers can, however, be difficult to find, and replacing a noncontractible intersection by a point discards its local homotopical information. Extensions of the nerve theorem address this limitation by weakening hypotheses or retaining more information about intersections \cite{BauerKerberRollRolle2023,ColinDeVerdiereGinotGoaoc2012,Ramras2026}. Nerve-like complexes are also central to topological data analysis, where filtrations track homological features across scales \cite{EdelsbrunnerLetscherZomorodian2002,ZomorodianCarlsson2005}. These considerations suggest retaining the entire intersection diagram, including both its local spaces and their inclusion maps.

We use this diagram to build categorical models. Unlike posets, small categories may have parallel morphisms and relations between their composites, allowing them to encode information with fewer objects in some examples. Finite and acyclic categories have also played an important role in the study of homotopy invariants: intrinsic, categorical versions of several continuous invariants have been developed for them, and these have been used to bound from above the Lusternik–Schnirelmann category, the homotopic distance, and related sectional invariants of their continuous counterparts \cite{Tanaka2018,MaciasVirgosMosquera2020,CarcaciaCamposEtAl2026MathSlovaca,CarcaciaCamposEtAl2026Filomat}.

Given compatible categorical models of the local spaces and their inclusions, we assemble them through the Grothendieck construction. Segal's open covering theorem, homotopy invariance of homotopy colimits and Thomason's theorem show that the resulting category has a classifying space of the same weak homotopy type as the covered space \cite{Segal1968,DuggerIsaksen2004,Thomason1979}. Diagrammatic and homotopy-colimit methods provide important precedents \cite{ZieglerZivaljevic1993,WelkerZieglerZivaljevic1999,FernandezMinian2016}; our focus is on assembling compatible local categories into a single global categorical model.

This procedure is recursive. Once categorical models are known for the members of a cover and their intersections, they can be assembled into a new categorical model of the total space. The local models need not be contractible or discrete, and the transition functors retain information about the maps among the local pieces. This makes it possible to encode nontrivial attaching maps, rather than only the incidence pattern of a good cover.

We consider two indexing posets. The first records all nonempty finite intersections of the cover. For point-finite covers, the second is the \emph{membership poset}, inspired by the finite-space construction of Sancho de Salas \cite{SanchoDeSalas2017}. It retains only the index sets that occur as exact membership patterns of points. Its diagram is a restriction of the full intersection diagram, yet its homotopy colimit still reconstructs the space. This can reduce the number of indexing objects without discarding the local spaces or their inclusion maps.

The classical nerve model and a componentwise variant arise as special cases. We also illustrate some applications of this theory by showing how to build categorical suspensions as a generalization of the non-Hausdorff suspension (see \cite[Section~2.7]{BARMAK}), an explicit seven-object model of the real projective plane, and finite acyclic models of bouquets and spaces formed by gluing Möbius bands along a common boundary.

Section~\ref{sec:categories-classifying-spaces} recalls the categorical constructions, and Section~\ref{sec:homotopy-colimits} collects the homotopy-colimit and open-cover results. Section~\ref{sec:categorical-cover-models} develops models over the full intersection poset; Section~\ref{sec:sancho-diagram} treats membership-pattern diagrams and their relation to the construction of Sancho de Salas. The appendix justifies homotopy invariance without objectwise cofibrancy assumptions.

All indexing categories, and all categories used to model spaces, are small. A category is finite if it has finitely many morphisms. We denote by $\cat$ the category of small categories and functors, by $\Top$ the category of topological spaces and continuous maps, and by $\Set$ the category of sets and functions. A set will occasionally be regarded as a discrete category. For diagrams with the same indexing category, we write $D\simeq_w E$ for a zigzag of objectwise weak homotopy equivalences.

\section{Small categories and their classifying spaces}
\label{sec:categories-classifying-spaces}
We recall the categorical constructions used throughout the article: classifying spaces, acyclic categories and the Grothendieck construction. General references are \cite[Chapters~10 and~11]{RichterCategoriesHomotopy} and \cite{Kozlov2008}.
\subsection{Nerves and classifying spaces}
Let $\Delta$ be the simplex category. Its objects are the finite ordered sets $[n]=\{0<1<\cdots<n\}$, and its morphisms are the order-preserving maps.
For a small category $\CC$, its nerve is the simplicial set $\N\CC\co\Delta^{\op}\to\Set$ whose $n$-simplices are the strings of composable morphisms
$$
c_0\xrightarrow{f_1}c_1\xrightarrow{f_2}\cdots
\xrightarrow{f_n}c_n.
$$
Equivalently,
$$
(\N\CC)_n=\coprod_{c_0,\ldots,c_n\in\Obj(\CC)}\Hom_{\CC}(c_0,c_1)\times\cdots\times \Hom_{\CC}(c_{n-1},c_n).
$$
The inner face maps compose consecutive morphisms, the two outer face maps remove the first or last object, and the degeneracy maps insert identities.

The \emph{classifying space} of $\CC$ is the geometric realization $B\CC=|\N\CC|$. Every functor $F\co\CC\to\DD$ induces a continuous map $BF\co B\CC\to B\DD$. A natural transformation between functors induces a homotopy between the corresponding maps of classifying spaces \cite{LeeHomotopyFunctors,LeeHomotopyFunctorsErratum,RichterCategoriesHomotopy}.

If $\CC$ has an initial or terminal object, then $B\CC$ is contractible. We will also use the canonical homeomorphism $B(\CC^{\op})\cong B\CC.$
A poset $P$ is regarded as a category with one morphism $p\to q$ whenever $p\leq q$. The nondegenerate simplices of $\N P$ are the strict chains $p_0<p_1<\cdots<p_n$, so $BP$ is the realization of the order complex of $P$.

Every finite $T_0$-space determines a finite poset through its specialization order, and every finite poset determines a finite $T_0$-space with its Alexandrov topology. McCord proved that a finite $T_0$-space and the classifying space of its associated poset have the same weak homotopy type \cite{McCord1966}.

\subsection{Acyclic categories}
A small category $\CC$ is \emph{acyclic} if every endomorphism is an identity and, for distinct objects $c,d$, the existence of a morphism $c\to d$ implies that there is no morphism $d\to c$. Every poset is acyclic, and an acyclic category is a poset precisely when there is at most one morphism between each pair of objects. We follow the terminology of \cite[Chapter~10]{Kozlov2008}; see also \cite{Tanaka2019}.

Acyclic categories may have parallel morphisms. For example, the Kronecker category
$$
x\overset{f}{\underset{g}{\rightrightarrows}}y
$$
is acyclic but is not a poset. Its classifying space is homeomorphic to $S^1$. Thus, parallel morphisms can encode homotopical information that would require additional objects in a poset model.

\subsection{The Grothendieck construction}
Let $\Fc\co\I\to\cat$ be a functor. Its \emph{Grothendieck construction}, denoted by $\int_{\I}\Fc$, has objects $(i,x)$, where $i\in\Obj(\I)$ and $x\in\Obj(\Fc(i))$. A morphism $(i,x)\to(j,y)$ is a pair $(u,\alpha)$, where $u\co i\to j$ is a morphism of $\I$ and $\alpha\co\Fc(u)(x)\to y$ is a morphism of $\Fc(j)$. Composition is given by
$$
(v,\beta)\circ(u,\alpha)=\bigl(vu,\beta\circ\Fc(v)(\alpha)\bigr).
$$
There is a canonical projection $\pi_{\Fc}\co\int_{\I}\Fc\to\I$, defined by $\pi_{\Fc}(i,x)=i$ and $\pi_{\Fc}(u,\alpha)=u$. We use the covariant convention of \cite{Thomason1979,FernandezMinian2016}.

If $\Fc$ is constant with value $\CC$, then $\int_{\I}\Fc\cong\I\times\CC$. In particular, if the constant value is the terminal category $\mathbf{1}$, then
$$
\int_{\I}\mathbf{1}\cong\I.
$$
If $P$ is a poset and $\Fc\co P\to\cat$ takes values in posets, then $\int_P\Fc$ is the poset whose elements are the pairs $(p,x)$, with
$$
(p,x)\leq(q,y)
\quad\Longleftrightarrow\quad
p\leq q
\text{ and }
\Fc(p\leq q)(x)\leq y.
$$
The following observation will be used to obtain acyclic models from acyclic local categories.
\begin{prop}
\label{prop:grothendieck-acyclic}
Let $\I$ be an acyclic category and let $\Fc\co\I\to\cat$ take values in acyclic categories. Then $\int_{\I}\Fc$ is acyclic. In particular, the conclusion holds when $\I$ is a poset. Moreover, if $\I$ is a poset and $\Fc$ takes values in posets, then $\int_{\I}\Fc$ is a poset.
\end{prop}
\begin{proof}
Let $(u,\alpha)\co(i,x)\to(i,x)$ be an endomorphism in $\int_{\I}\Fc$. Its image under the canonical projection $\pi_{\Fc}\co\int_{\I}\Fc\to\I$ is the endomorphism $u\co i\to i$. Since $\I$ is acyclic, $u=\id_i$. Therefore, $(u,\alpha)$ is determined by an endomorphism $\alpha\co x\to x$ in $\Fc(i)$, which must also be an identity. Hence every endomorphism of $\int_{\I}\Fc$ is an identity.
Now suppose that there are morphisms
\[
(i,x)\longrightarrow(j,y)
\qquad\text{and}\qquad
(j,y)\longrightarrow(i,x)
\]
between distinct objects. Their images in $\I$ give morphisms $i\to j$ and $j\to i$. Since $\I$ is acyclic, this is possible only if $i=j$. Both morphisms then lie in the fiber $\Fc(i)$, so the acyclicity of $\Fc(i)$ implies $x=y$. This contradicts the assumption that $(i,x)$ and $(j,y)$ are distinct. Thus, $\int_{\I}\Fc$ is acyclic.

Finally, suppose that $\I$ and all the categories $\Fc(i)$ are posets. A morphism $(i,x)\to(j,y)$ is determined by a morphism $u\co i\to j$ in $\I$ and a morphism $\Fc(u)(x)\to y$ in $\Fc(j)$. Both are unique whenever they exist. Hence there is at most one morphism between any two objects of $\int_{\I}\Fc$, and the Grothendieck construction is a poset.
\end{proof}
\section{Homotopy colimits and reconstruction from open covers}
\label{sec:homotopy-colimits}
Homotopy colimits provide homotopically meaningful replacements for ordinary colimits. We recall their homotopy invariance, Thomason's theorem and two results reconstructing a space from a diagram of open subsets. For comparison results concerning homotopy colimits of diagrams over posets, we refer also to \cite{ZieglerZivaljevic1993, WelkerZieglerZivaljevic1999}.
\subsection{Homotopy colimits}
A simplicial space is a functor $Z\co\Delta^{\op}\to\Top$. Its geometric realization is obtained by gluing the spaces $Z_n\times\Delta^n$ according to the face and degeneracy maps:
$$
|Z| = \coprod_{n\geq0}Z_n\times\Delta^n\big/\sim.
$$
A simplicial set may be regarded as a simplicial space by giving each set of simplices the discrete topology. With this convention, the classifying space $B\CC$ introduced in Section~\ref{sec:categories-classifying-spaces} is the realization of the simplicial space $\N\CC$.
Let $\I$ be a small category and let $X\co\I\to\Top$ be a functor. Its simplicial replacement is the simplicial space
$$
\operatorname{srep}(X)_n=\coprod_{i_0\to i_1\to\cdots\to i_n}X(i_0).
$$
The face maps compose consecutive morphisms or apply a structure map of $X$, while the degeneracy maps insert identities. We define the homotopy colimit of $X$ by the Bousfield--Kan formula
$$
\operatorname{hocolim}_{\I}X= \left|\operatorname{srep}(X)\right|.
$$
We refer to \cite{BousfieldKan1972,DuggerHocolimPrimer} for this construction and its basic properties.
The property needed throughout the article is homotopy invariance.
\begin{teo}[Homotopy invariance]
\label{thm:hocolim-invariance}
Let $X,Y\co\I\to\Top$ be diagrams and let $\alpha\co X\Rightarrow Y$ be an objectwise weak homotopy equivalence. Then the induced map
$$
\operatorname{hocolim}_{\I}X \longrightarrow \operatorname{hocolim}_{\I}Y
$$
is a weak homotopy equivalence.
\end{teo}
\begin{proof}
See Proposition~\ref{prop:hocolim-invariance-appendix}.
\end{proof}
It is therefore not enough to choose unrelated weak equivalences $X(i)\simeq_w Y(i)$ for the individual objects. The weak equivalences must form a natural transformation, or more generally a zigzag of objectwise weak equivalences in $\Top^{\I}$.
\subsection{Thomason's homotopy colimit theorem}
The Grothendieck construction provides a categorical model of the homotopy colimit of a category-valued diagram.
\begin{teo}[Thomason]
\label{thm:thomason}
Let $\Fc\co\I\to\cat$ be a functor. There is a natural weak homotopy equivalence
$$
\operatorname{hocolim}_{\I}B\Fc
\longrightarrow
B\left(\int_{\I}\Fc\right).
$$
\end{teo}
This is Thomason's homotopy colimit theorem \cite{Thomason1979}. At the simplicial level, Thomason constructs a natural weak equivalence between the homotopy colimit of $\N\Fc$ and the nerve of $\int_{\I}\Fc$. See also \cite{FernandezMinian2016} for diagrams indexed by posets and related combinatorial formulations.

Combining Thomason's theorem with homotopy invariance we obtain the following result, which will be used repeatedly.
\begin{cor}
\label{cor:categorical-diagram-model}
Let $X\co\I\to\Top$ and $\Fc\co\I\to\cat$ be diagrams. If $X$ and $B\Fc$ are connected by a zigzag of objectwise weak homotopy equivalences, then
$$
\operatorname{hocolim}_{\I}X
\simeq_w
B\left(\int_{\I}\Fc\right).
$$
In particular, the conclusion holds if there is an objectwise weak equivalence $X\Rightarrow B\Fc$ or $B\Fc\Rightarrow X$.
\end{cor}
\begin{proof}
Homotopy invariance identifies the homotopy colimits of $X$ and $B\Fc$, and Theorem~\ref{thm:thomason} identifies the latter with the classifying space of the Grothendieck construction.
\end{proof}
\subsection{Diagrams of open subsets}
Let $\operatorname{Op}(X)$ denote the poset of open subsets of a topological space $X$, ordered by inclusion. A functor $V\co\I\to\operatorname{Op}(X)$ determines a diagram of spaces and a canonical map
$$
\varepsilon_V\co
\operatorname{hocolim}_{\I}V
\longrightarrow X.
$$
For $x\in X$, let $\I_x$ be the full subcategory of $\I$ whose objects are those $i$ for which $x\in V(i)$.
\begin{teo}[Segal's open covering theorem]
\label{thm:segal-open-covering}
Suppose that the open subsets $V(i)$ cover $X$ and that $B\I_x$ is contractible for every $x\in X$. Then the canonical map $\varepsilon_V\co\operatorname{hocolim}_{\I}V\to X$ is a weak homotopy equivalence.
\end{teo}
This is Segal's open covering theorem \cite{Segal1968}; we use the formulation and clarification of its proof given in \cite{DuggerSegalOpenCover}. A formulation in terms of singular complexes and homotopy colimits of simplicial sets is given by Lurie \cite[Theorem~A.3.1]{LurieHigherAlgebra}. A particularly useful sufficient condition is that every $\I_x$ has an initial or terminal object.
\subsection{The intersection diagram of a cover}
Let $\Uc=\{U_i\}_{i\in I}$ be an indexed open cover of $X$, and let $\mathcal P_f(I)$ be the poset of nonempty finite subsets of $I$, ordered by inclusion. For $\sigma\in\mathcal P_f(I)$, write \(U_\sigma=\bigcap_{i\in\sigma}U_i\).
Since $\sigma\subseteq\tau$ implies $U_\tau\subseteq U_\sigma$, the intersections define a functor
$$
\Uc_{\cap}\co\mathcal P_f(I)^{\op}\longrightarrow\Top,
\qquad
\sigma\longmapsto U_\sigma.
$$
Let $P_{\Uc}\subseteq\mathcal P_f(I)$ be the subposet of those $\sigma$ for which $U_\sigma$ is nonempty. The empty intersections contribute no summands to the simplicial replacement. Thus, after the canonical identification of their summands, the simplicial replacements over $P_f(I)^{\op}$ and $P_{\Uc}^{\op}$ are the same, and so are their realizations.
\begin{teo}[Open-cover reconstruction]
\label{thm:open-cover-hocolim}
Let $\Uc=\{U_i\}_{i\in I}$ be an open cover of $X$. Then the canonical augmentation
$$
\varepsilon_{\Uc}\co
\operatorname{hocolim}_{P_{\Uc}^{\op}}U_\sigma
\longrightarrow X
$$
is a weak homotopy equivalence. If $\Uc$ is numerable, then $\varepsilon_{\Uc}$ is a homotopy equivalence.
\end{teo}
\begin{proof}
We apply Theorem~\ref{thm:segal-open-covering} to the intersection diagram
$$
\Uc_{\cap}\co P_{\Uc}^{\op}\longrightarrow\operatorname{Op}(X),
\qquad
\sigma\longmapsto U_\sigma.
$$
Fix $x\in X$ and set
$$
I_x=\{i\in I:x\in U_i\}.
$$
The full subcategory of $P_{\Uc}^{\op}$ formed by the objects $\sigma$ such that $x\in U_\sigma$ is precisely $P_f(I_x)^{\op},$ where $P_f(I_x)$ is the poset of nonempty finite subsets of $I_x$.

The poset $P_f(I_x)$ is filtered: for every $\sigma,\tau\in P_f(I_x)$, the finite subset $\sigma\cup\tau$ is an upper bound. Hence $P_f(I_x)^{\op}$ is cofiltered and its classifying space is contractible. Theorem~\ref{thm:segal-open-covering} therefore implies that $\varepsilon_{\Uc}$ is a weak homotopy equivalence. When $\Uc$ is finite, the same conclusion follows immediately because $I_x$ is an initial object of $P_f(I_x)^{\op}$.
If $\Uc$ is numerable, a subordinate partition of unity yields a homotopy inverse for the projection from the homotopy colimit projection; see \cite{Segal1968,BauerKerberRollRolle2023}.
\end{proof}
\begin{remark}
No contractibility assumption is imposed on the intersections $U_\sigma$. Their homotopy types and inclusion maps remain part of the diagram. Contractibility enters only when the intersection spaces are replaced by simpler objects, as in the classical nerve theorem.
\end{remark}
\section{Categorical models subordinate to open covers}
\label{sec:categorical-cover-models}
We now replace the intersections of an open cover by compatible categorical models and assemble them through the Grothendieck construction.
\begin{definition} \label{def:categorical-cover-model}
Let $\Uc=\{U_i\}_{i\in I}$ be an indexed open cover of $X$. A \emph{categorical cover model} subordinate to $\Uc$ consists of a functor $\Fc\co P_{\Uc}^{\op}\to\cat$ together with a zigzag of objectwise weak homotopy equivalences of diagrams
\[
\Uc_{\cap}|_{P_{\Uc}^{\op}} \simeq_w B\Fc
\]
in $\Top^{P_{\Uc}^{\op}}$.
\end{definition}
\begin{definition}
The \emph{Thomason model} associated with a categorical cover model $(\Uc,\Fc)$ is
$$
\operatorname{Th}(\Uc,\Fc)
=
\int_{P_{\Uc}^{\op}}\Fc.
$$
\end{definition}
\begin{teo}[Categorical reconstruction]
\label{thm:categorical-reconstruction}
Let $(\Uc,\Fc)$ be a categorical cover model of a topological space $X$. Then
$$
B\operatorname{Th}(\Uc,\Fc)\simeq_w X.
$$
More precisely, there is a zigzag of weak homotopy equivalences
\[
X \xleftarrow{\simeq_w} \operatorname{hocolim}_{P_{\Uc}^{\op}}U_\sigma \simeq_w \operatorname{hocolim}_{P_{\Uc}^{\op}}B\Fc(\sigma) \xrightarrow{\simeq_w} B\left(\int_{P_{\Uc}^{\op}}\Fc\right).
\]
If every $\Fc(\sigma)$ is acyclic, then $\operatorname{Th}(\Uc,\Fc)$ is acyclic. If, moreover, $P_{\Uc}$ and all the categories $\Fc(\sigma)$ are finite, then $\operatorname{Th}(\Uc,\Fc)$ is a finite acyclic category.
\end{teo}
\begin{proof}
The first comparison is the open-cover reconstruction of Theorem~\ref{thm:open-cover-hocolim}. The second follows from homotopy invariance of homotopy colimits, and the third is Thomason's comparison from Theorem~\ref{thm:thomason}.

When every $\Fc(\sigma)$ is acyclic, the Grothendieck construction is acyclic by Proposition~\ref{prop:grothendieck-acyclic}, since $P_{\Uc}^{\op}$ is a poset. The finiteness assertion is immediate.
\end{proof}
\subsection{Recovery of classical constructions}
\label{subsec:classical-constructions}
Theorem~\ref{thm:categorical-reconstruction} contains the classical nerve construction and its componentwise refinement as particular cases. These examples also explain why it is natural to compare the intersection diagram with a categorical diagram by maps of the form $\Uc_{\cap}\Rightarrow B\Fc$.
\begin{cor}[Good covers]
\label{cor:good-cover}
Let $\Uc=\{U_i\}_{i\in I}$ be an open cover of $X$. If every nonempty finite intersection $U_\sigma$ is weakly contractible, then
\[
BP_{\Uc}\simeq_w X.
\]
If, moreover, $\Uc$ is numerable and every nonempty finite intersection is contractible, then
\[
BP_{\Uc}\simeq X.
\]
\end{cor}
\begin{proof}
For the first assertion, consider the constant functor $\mathbf{1}\co P_{\Uc}^{\op}\to\cat$ with value the terminal category. The unique maps $U_\sigma\to B\mathbf{1}=*$ form an objectwise weak homotopy equivalence
\[
\Uc_{\cap}|_{P_{\Uc}^{\op}}
\Longrightarrow
B\mathbf{1}.
\]
Theorem~\ref{thm:categorical-reconstruction} therefore gives
\[
X\simeq_w B\left(\int_{P_{\Uc}^{\op}}\mathbf{1}\right).
\]
Since
\[
\int_{P_{\Uc}^{\op}}\mathbf{1} \cong P_{\Uc}^{\op}
\]
and $B(P_{\Uc}^{\op})\cong BP_{\Uc}$, the first assertion follows.
For the second assertion, $P_{\Uc}$ is the face poset of the nerve $N(\Uc)$, so $BP_{\Uc}$ is homeomorphic to the barycentric subdivision of $|N(\Uc)|$. Hence, the numerable nerve theorem gives $|N(\Uc)|\simeq X$; see \cite[Proposition~4.1]{Segal1968} or \cite[Corollary~4G.3]{Hatcher2002}.
\end{proof}
The next consequence keeps the connected components of the intersections. This allows several simplices to have the same set of vertices.
\begin{cor}[Componentwise good covers]
\label{cor:componentwise-good-cover}
Let $\Uc=\{U_i\}_{i\in I}$ be an open cover of $X$. Suppose that every nonempty finite intersection $U_\sigma$ is a disjoint union of open weakly contractible path components. Define $\Pi_{\Uc}\co P_{\Uc}^{\op}\to\cat$ by
\[
\Pi_{\Uc}(\sigma)=\pi_0(U_\sigma),
\]
where $\pi_0(U_\sigma)$ denotes the set of path components of $U_\sigma$, regarded as a discrete category. Then
\[
B\left(\int_{P_{\Uc}^{\op}}\Pi_{\Uc}\right) \simeq_w X.
\]
Moreover, $\int_{P_{\Uc}^{\op}}\Pi_{\Uc}$ is a poset.
\end{cor}
\begin{proof}
For each $\sigma\in P_{\Uc}$, let
\[
q_\sigma\co U_\sigma\longrightarrow\pi_0(U_\sigma)
\]
send every point to its path component. Since the path components of $U_\sigma$ are open, $q_\sigma$ is continuous when $\pi_0(U_\sigma)$ is endowed with the discrete topology.
The maps $q_\sigma$ form a natural transformation
\[
\Uc_{\cap}|_{P_{\Uc}^{\op}} \Longrightarrow B\Pi_{\Uc}.
\]
Indeed, an inclusion $U_\tau\hookrightarrow U_\sigma$ sends each path component of $U_\tau$ into a unique path component of $U_\sigma$.
Each $q_\sigma$ is a weak homotopy equivalence. It induces a bijection on path components, and its restriction to every path component has contractible target and weakly contractible source. Therefore, Theorem~\ref{thm:categorical-reconstruction} gives
\[
B\left(\int_{P_{\Uc}^{\op}}\Pi_{\Uc}\right) \simeq_w X.
\]
Finally, $P_{\Uc}^{\op}$ is a poset and every $\Pi_{\Uc}(\sigma)$ is a discrete poset. Hence their Grothendieck construction is a poset.
\end{proof}
The objects of $\int_{P_{\Uc}^{\op}}\Pi_{\Uc}$ are the pairs $(\sigma,C)$, where $C$ is a path connected component of $U_\sigma$. There is a morphism $(\tau,D)\to(\sigma,C)$ precisely when $\sigma\subseteq\tau$ and $D\subseteq C$. Thus, the Grothendieck construction records both the components of finite intersections and their incidence under the omission of indices.
\begin{remark}
In both corollaries, the canonical comparison points from the intersection diagram to the categorical diagram. For good covers, each intersection is collapsed to a point. For componentwise good covers, each point is sent to its path connected component. These maps are natural without requiring choices of basepoints.
\end{remark}
\subsection{Categorical suspensions}
\label{subsec:categorical-suspension}
We use the reconstruction theorem to build categorical models of suspensions. Let $\mathbb V$ be the poset with relations $0<-$ and $0<+$, with no relation between $-$ and $+$.
For a small category $\CC$, define $\Si_{\CC}\co\mathbb V\to\cat$ by
\[
\Si_{\CC}(0)=\CC,
\qquad
\Si_{\CC}(-)=\Si_{\CC}(+)=\mathbf{1},
\]
where both structure functors $\CC\to\mathbf{1}$ are unique.
The \emph{categorical suspension} of $\CC$ is the Grothendieck construction
\[
\Sigma_{\cat}\CC=\int_{\mathbb V}\Si_{\CC}.
\]
Explicitly, it is obtained from $\CC$ by adjoining two objects $v_-$ and $v_+$ and unique morphisms
\[
c\longrightarrow v_-,
\qquad
c\longrightarrow v_+
\]
for every $c\in\Obj(\CC)$. There are no morphisms between $v_-$ and $v_+$. In other words, every object of $\CC$ has a unique morphism to each of them.
If $\CC$ is acyclic, then $\Sigma_{\cat}\CC$ is acyclic by Proposition~\ref{prop:grothendieck-acyclic}. If $\CC$ is finite, then $\Sigma_{\cat}\CC$ is finite.
\begin{teo}[Categorical suspension theorem]
\label{thm:categorical-suspension}
For every small category $\CC$, there is a weak homotopy equivalence
\[
B\bigl(\Sigma_{\cat}\CC\bigr)\simeq_w\Sigma B\CC,
\]
where $\Sigma B\CC$ denotes the unreduced suspension.
\end{teo}
\begin{proof}
Put $Y=B\CC$ and cover $\Sigma Y$ by the two open cone neighborhoods
\[
U_-=\{[y,t]:t<2/3\},
\qquad
U_+=\{[y,t]:t>-2/3\}.
\]
Both sets are contractible, while $U_-\cap U_+\cong Y\times(-2/3,2/3)$.
The intersection diagram of this cover is indexed by $\mathbb V$. Projection onto $Y$ on the intersection, together with the unique maps $U_-\to*$ and $U_+\to*$, gives an objectwise homotopy equivalence from the intersection diagram to $B\Si_{\CC}$. Thus, Theorem~\ref{thm:categorical-reconstruction} gives
\[
B\bigl(\Sigma_{\cat}\CC\bigr) =B\left(\int_{\mathbb V}\Si_{\CC}\right) \simeq_w \Sigma B\CC.
\]
\end{proof}
\begin{remark}
Dually, one may adjoin two initial cone objects. The resulting category is the opposite of the terminal categorical suspension of $\CC^{\op}$ and has the same classifying-space.
\end{remark}
\begin{cor}
\label{cor:suspension-categorical-model}
Let $X$ be a topological space and let $\CC$ be a finite acyclic category such that $X\simeq_w B\CC$. Then $\Sigma_{\cat}\CC$ is a finite acyclic categorical model of $\Sigma X$.
\end{cor}
\begin{proof}
Consider the two-cone cover of $\Sigma X$. Its intersection is homotopy equivalent to $X$, while both members are contractible. The zigzag $X\simeq_w B\CC$, together with the unique maps to the one-point space, gives a zigzag of objectwise weak equivalences between the intersection diagram and $B\Si_{\CC}$. The conclusion follows from Theorem~\ref{thm:categorical-reconstruction}.
\end{proof}
\begin{remark} When $\CC$ is a poset, $\Sigma_{\cat}\CC$ is its non-Hausdorff suspension, obtained by adjoining two incomparable maximal elements; see \cite[Section~2.7]{BARMAK}. Thus, categorical suspension extends this construction from posets to small categories. \end{remark}
\begin{example}[A categorical model of the sphere]
\label{ex:sphere}
Let $\mathcal K$ be the Kronecker category
\[
x\overset{f}{\underset{g}{\rightrightarrows}}y,
\]
whose classifying space is homeomorphic to $S^1$. Its categorical suspension $\mathcal C_{S^2}=\Sigma_{\cat}\mathcal K$ is generated by
\[
\begin{tikzcd}
& a & \\
x \arrow[ru, "a_x"] \arrow[rr, "f", bend left] \arrow[rr, "g"', bend right] \arrow[rd, "b_x"'] & & y \arrow[lu, "a_y"'] \arrow[ld, "b_y"] \\
& b &
\end{tikzcd}
\]
subject to
\[
a_yf=a_yg=a_x,
\qquad
b_yf=b_yg=b_x.
\]
The objects $a$ and $b$ are the two terminal cone objects. By Theorem~\ref{thm:categorical-suspension},
\[
B\mathcal C_{S^2}
\simeq_w
\Sigma B\mathcal K
\cong
S^2.
\]
Thus, $\mathcal C_{S^2}$ is a four-object acyclic categorical model of $S^2$.
\end{example}
\begin{remark}
Starting with the Kronecker category and iterating the categorical suspension, we obtain an acyclic categorical model of $S^n$ with $2n$ objects. In contrast, every finite $T_0$-space modeling $S^n$ has at least $2n+2$ points \cite{BarmakMinian2007}.
\end{remark}
\subsection{A finite categorical model of the real projective plane}
\label{subsec:projective-plane-model}
Write $\mathbb{RP}^{2}=U\cup V$, where $U$ is an open neighborhood of a Möbius band, $V$ is an open disk and $U\cap V$ is an open annulus. Thus, $U\simeq S^1$, $V\simeq *$ and $U\cap V\simeq S^1$. We choose the cover and these equivalences compatibly, so that the inclusion $U\cap V\hookrightarrow U$ corresponds strictly to a degree-two map between the chosen models of the circles.
Let $P$ be the crown poset
$$
x\longrightarrow a\longleftarrow y,
\qquad
x\longrightarrow b\longleftarrow y,
$$
and let $\mathcal K$ be the Kronecker category
$$
u\overset{f}{\underset{g}{\rightrightarrows}}v.
$$
Both $BP$ and $B\mathcal K$ are homeomorphic to $S^1$. Denote the four nonidentity morphisms of $P$ by
$$
f_1\colon x\to a,\qquad
g_1\colon x\to b,\qquad
g_2\colon y\to a,\qquad
f_2\colon y\to b.
$$
Define $d\colon P\to\mathcal K$ by
$$
d(x)=d(y)=u,\qquad d(a)=d(b)=v,
$$
and
$$
d(f_1)=d(f_2)=f,\qquad d(g_1)=d(g_2)=g.
$$
The circuit of $BP$ is sent to the square of a generator of $\pi_1(B\mathcal K)$, so $Bd\colon BP\to B\mathcal K$ has degree two.
Let $\iota_-\co U\cap V\hookrightarrow U$ and $\iota_+\co U\cap V\hookrightarrow V$ be the inclusions. The map $Bd\co BP\to B\mathcal K$ is a twofold covering. Choose a homotopy equivalence $\eta_-\co U\to B\mathcal K$ by retracting $U$ onto the core circle of the Möbius band. Since $\iota_-$ has degree two relative to the core, $(\eta_-\iota_-)_*$ maps $\pi_1(U\cap V)$ onto $2\mathbb Z$, the image of $(Bd)_*$. The covering-space lifting criterion \cite[Proposition~1.33]{Hatcher2002} therefore gives a lift $\eta_0\co U\cap V\to BP$ satisfying $Bd\,\eta_0=\eta_-\iota_-$.

Both $U\cap V$ and $BP$ have the homotopy type of a circle. Since $(Bd)_*$ and $(\eta_-\iota_-)_*$ both have degree two up to sign, $\eta_0$ has degree one up to sign and is a homotopy equivalence. Together with the unique map $\eta_+\co V\to*$, the maps $\eta_0,\eta_-,\eta_+$ define an objectwise homotopy equivalence of strict diagrams: the square involving $\iota_-$ commutes by construction, and the one involving $\iota_+$ commutes because both maps have target $*$.

Let $\mathcal G_{\mathbb{RP}^{2}}$ be the Grothendieck construction of this diagram. It has seven objects and is generated by
$$
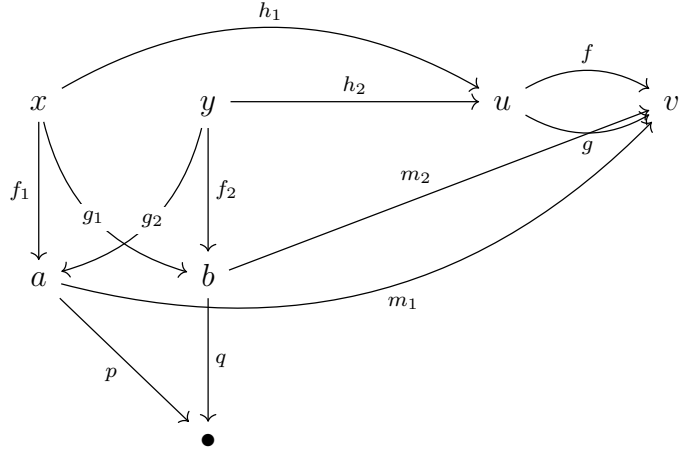
\begin{tikzcd}[column sep=4em, row sep=4em]
x
\arrow[d, "f_1"']
\arrow[rd, "g_1" description, bend right]
\arrow[rrr, "h_1", bend left]
&
y
\arrow[ld, "g_2" description, bend left]
\arrow[d, "f_2"]
\arrow[rr, "h_2"]
&&
u
\arrow[r, "f", bend left]
\arrow[r, "g"', bend right]
&
v
\\
a
\arrow[rrrru, "m_1"', bend right]
\arrow[rd, "p"']
&
b
\arrow[rrru, "m_2"]
\arrow[d, "q"]
&&&
\\
&
\bullet
&&&
\end{tikzcd}
$$
subject to
$$
m_1f_1=fh_1,\qquad
m_2g_1=gh_1,\qquad
m_1g_2=gh_2,\qquad
m_2f_2=fh_2,
$$
and
$$
pf_1=qg_1,\qquad
pg_2=qf_2.
$$
The first four relations encode the degree-two map, while the last two come from the unique functor $P\to\mathbf{1}$.
Theorem~\ref{thm:categorical-reconstruction} now gives
$$
B\mathcal G_{\mathbb{RP}^{2}}\simeq_w\mathbb{RP}^{2}.
$$

This construction gives a seven-object acyclic categorical model of $\mathbb{RP}^{2}$. For comparison, Barmak describes a finite $T_0$-space model with thirteen points \cite[Example~7.1.1]{BARMAK}, and Cianci and Ottina proved that no finite $T_0$-space model has fewer points \cite{CianciOttina2018}. Tanaka gives a smaller acyclic categorical model with only three objects \cite[Example~3.8]{Tanaka2018}. Our model is not intended to be minimal; its purpose is to exhibit how a categorical model can be assembled from an open cover and compatible local models.
\section{Categorical models from membership-pattern diagrams}
\label{sec:sancho-diagram}
We now consider the smaller diagram determined by the membership patterns realized by points of the space. For finite covers, this is the finite-space construction of Sancho de Salas \cite{SanchoDeSalas2017}. The same construction applies to point-finite covers.
Let $\Uc=\{U_i\}_{i\in I}$ be a point-finite open cover of $X$. For $x\in X$, set
\[
I_x=\{i\in I:x\in U_i\},
\qquad
U_x=\bigcap_{i\in I_x}U_i.
\]
Point-finiteness ensures that $I_x$ is finite, so $U_x$ is open. Moreover, $I_x$ is nonempty and $x\in U_x$.
Define $x\sim_{\Uc}y$ if $I_x=I_y$, equivalently if $U_x=U_y$, and write $q_x$ for the class of $x$. The \emph{membership poset} of $\Uc$ is
\[
Q_{\Uc}=X/{\sim_{\Uc}},
\]
ordered by
\[
q_x\leq q_y
\quad\Longleftrightarrow\quad
I_x\subseteq I_y
\quad\Longleftrightarrow\quad
U_x\supseteq U_y.
\]
For $q=q_x$, write $U_q=U_x$. The assignment $q\mapsto U_q$ defines the \emph{membership-pattern diagram}
\[
\Dc_{\Uc}\co Q_{\Uc}^{\op}\longrightarrow\operatorname{Op}(X).
\]
The map $q_x\mapsto I_x$ identifies $Q_{\Uc}$ with the subposet
\[
\{I_x:x\in X\}\subseteq P_{\Uc}.
\]
Under this identification,
\[
\Dc_{\Uc} = \Uc_{\cap}|_{Q_{\Uc}^{\op}}.
\]
Thus, the membership-pattern diagram is the restriction of the full intersection diagram to those families of indices which occur as the exact membership pattern of a point. When $\Uc$ is finite, $Q_{\Uc}$ is a finite poset. This construction is also related to McCord's use of point-finite, basis-like open covers to obtain weak homotopy models \cite{McCord1967}.
\subsection{Reconstruction from the membership-pattern diagram}
Since $x\in U_{q_x}$ for every $x\in X$, the sets $U_q$ cover $X$. Their inclusions induce an augmentation
\[
\varepsilon_{\Uc}\co\operatorname{hocolim}_{Q_{\Uc}^{\op}}U_q \longrightarrow X.
\]
\begin{teo}[Membership-pattern reconstruction]
\label{thm:sancho-reconstruction}
Let $\Uc$ be a point-finite open cover of $X$. Then $\varepsilon_{\Uc}$ is a weak homotopy equivalence.
\end{teo}
\begin{proof}
We apply Theorem~\ref{thm:segal-open-covering} to $\Dc_{\Uc}$. Fix $x\in X$, and let $(Q_{\Uc}^{\op})_x$ be the full subcategory formed by the objects $q$ such that $x\in U_q$.

The object $q_x$ belongs to $(Q_{\Uc}^{\op})_x$. If $q=q_y$ and $x\in U_q=U_y$, then $I_y\subseteq I_x$. Hence $q_y\leq q_x$ in $Q_{\Uc}$, so there is a unique morphism $q_x\to q_y=q$ in $Q_{\Uc}^{\op}$. Therefore, $q_x$ is initial in $(Q_{\Uc}^{\op})_x$, and its classifying space is contractible. The result follows from Theorem~\ref{thm:segal-open-covering}.
\end{proof}
\subsection{Categorical membership models}
A \emph{categorical membership model} subordinate to $\Uc$ consists of a functor
\[
\Fc\co Q_{\Uc}^{\op}\longrightarrow\cat
\]
such that $\Dc_{\Uc}$ and $B\Fc$ are connected by a zigzag of objectwise weak homotopy equivalences in $\Top^{Q_{\Uc}^{\op}}$. The associated global category is
\[
\operatorname{MTh}(\Uc,\Fc)= \int_{Q_{\Uc}^{\op}}\Fc.
\]
\begin{teo}[Categorical membership reconstruction]
\label{thm:categorical-sancho}
Let $\Fc\co Q_{\Uc}^{\op}\to\cat$ be a categorical membership model. Then
\[
B\operatorname{MTh}(\Uc,\Fc)\simeq_w X.
\]
More precisely, there is a zigzag of weak homotopy equivalences \[ X \xleftarrow{\simeq_w} \operatorname{hocolim}_{Q_{\Uc}^{\op}}U_q \simeq_w \operatorname{hocolim}_{Q_{\Uc}^{\op}}B\Fc(q) \xrightarrow{\simeq_w} B\left(\int_{Q_{\Uc}^{\op}}\Fc\right). \]
Moreover, if every $\Fc(q)$ is acyclic, then $\operatorname{MTh}(\Uc,\Fc)$ is acyclic. If, in addition, $Q_{\Uc}$ and all the categories $\Fc(q)$ are finite, then $\operatorname{MTh}(\Uc,\Fc)$ is a finite acyclic category.
\end{teo}
\begin{proof}
Theorem~\ref{thm:sancho-reconstruction}, homotopy invariance and Thomason's theorem give the zigzag of weak homotopy equivalences. As before, the assertions about acyclicity and finiteness follow from Proposition~\ref{prop:grothendieck-acyclic}.
\end{proof}
Every categorical model over the full intersection poset restricts to one over the membership poset.
\begin{prop}
\label{prop:restriction-membership-model}
Let $\Fc\co P_{\Uc}^{\op}\to\cat$ be a categorical cover model subordinate to $\Uc$. Then
\[
\Fc|_{Q_{\Uc}^{\op}} \co Q_{\Uc}^{\op}\longrightarrow\cat
\]
is a categorical membership model. Consequently,
\[
B\left(\int_{Q_{\Uc}^{\op}} \Fc|_{Q_{\Uc}^{\op}} \right) \simeq_w X.
\]
\end{prop}
\begin{proof}
The membership-pattern diagram is the restriction of the intersection diagram:
\[
\Dc_{\Uc}= \Uc_{\cap}|_{Q_{\Uc}^{\op}}.
\]
Restricting the zigzag of objectwise weak equivalences between $\Uc_{\cap}$ and $B\Fc$ therefore gives a zigzag between $\Dc_{\Uc}$ and $B(\Fc|_{Q_{\Uc}^{\op}})$. The conclusion follows from Theorem~\ref{thm:categorical-sancho}.
\end{proof}
\begin{remark}
Proposition~\ref{prop:restriction-membership-model} connects the two reconstruction procedures. The full intersection diagram preserves one object for every nonempty formal intersection, whereas the membership-pattern diagram retains only the intersections of the form $U_x$. Both reconstruct $X$, but the second indexing poset may be smaller.
\end{remark}
\subsection{Recovery of the Sancho de Salas model}
\begin{cor}
\label{cor:sancho-contractible}
If every $U_q$ is weakly contractible, then
\[
BQ_{\Uc}\simeq_w X.
\]
\end{cor}
\begin{proof}
Take the constant functor $\mathbf 1\co Q_{\Uc}^{\op}\to\cat$. The maps $U_q\to B\mathbf 1=*$ form an objectwise weak equivalence, while
\[
\int_{Q_{\Uc}^{\op}}\mathbf 1\cong Q_{\Uc}^{\op}.
\]
The result follows from Theorem~\ref{thm:categorical-sancho} and the homeomorphism $B(Q_{\Uc}^{\op})\cong BQ_{\Uc}$.
\end{proof}
\begin{remark}
When $\Uc$ is finite, $Q_{\Uc}$ is the finite $T_0$-space associated with the cover by Sancho de Salas. By McCord's theorem,
\[
BQ_{\Uc}\simeq_w Q_{\Uc}.
\]
Therefore, Corollary~\ref{cor:sancho-contractible} is equivalent, at the level of weak homotopy types, to the Sancho de Salas approximation $X\simeq_w Q_{\Uc}$ \cite{McCord1966,SanchoDeSalas2017}. The original result additionally identifies the quotient map $X\to Q_{\Uc}$ as a weak homotopy equivalence.
\end{remark}
There is also a componentwise version.
\begin{cor}
\label{cor:sancho-componentwise}
Suppose that every $U_q$ is a disjoint union of open weakly contractible path components. Define
\[
\Pi_{\Uc}^{\mathrm M}\co Q_{\Uc}^{\op}\longrightarrow\cat,
\qquad
\Pi_{\Uc}^{\mathrm M}(q)=\pi_0(U_q),
\]
where $\pi_0(U_q)$ is regarded as a discrete category. Then
\[
B\left( \int_{Q_{\Uc}^{\op}}\Pi_{\Uc}^{\mathrm M}\right) \simeq_w X.
\]
If $Q_{\Uc}$ is finite and every $U_q$ has finitely many path components, the Grothendieck construction is a finite poset.
\end{cor}
\begin{proof}
The quotient maps $U_q\to\pi_0(U_q)$ are continuous because the path components are open, and they are weak homotopy equivalences because those components are weakly contractible. They form a natural transformation $\Dc_{\Uc}\Rightarrow B\Pi_{\Uc}^{\mathrm M}$, so the result follows from Theorem~\ref{thm:categorical-sancho}. Since the indexing category and the values of $\Pi_{\Uc}^{\mathrm M}$ are posets, their Grothendieck construction is a poset.
\end{proof}
This is the membership-pattern analogue of the multinerve construction: it retains the path components of the realized intersections $U_q$, rather than those of every formal intersection.
\subsection{Examples}
\begin{example}
\label{ex:sancho-bouquet}
Let $m\geq2$ and let $X=\bigvee_{r=1}^{m}S^1_r$. Choose a contractible open neighborhood $W$ of the common vertex. For each $r$, let $U_r$ be an open neighborhood of $S^1_r\cup W$ which deformation retracts onto $S^1_r$ and is chosen so that every intersection involving at least two distinct members of the cover is equal to $W$.
The realized membership patterns are
\[
\{1\},\ldots,\{m\},
\qquad
\{1,\ldots,m\}.
\]
Hence the membership poset $Q_{\Uc}$ has objects $q_1,\ldots,q_m,q_*$ and relations $q_r<q_*$.
For each $r$, let $\mathcal K_r$ be a copy of the Kronecker category, and choose an object $z_r\in\mathcal K_r$. Define $\Fc\co Q_{\Uc}^{\op}\to\cat$ by
\[
\Fc(q_*)=\mathbf 1,
\qquad
\Fc(q_r)=\mathcal K_r,
\]
and send the morphism $q_*\to q_r$ in $Q_{\Uc}^{\op}$ to the functor $\mathbf 1\to\mathcal K_r$ selecting $z_r$.
The local equivalences can be chosen compatibly. Indeed, choose a weak homotopy equivalence
\[
\eta_r\co U_r\longrightarrow B\mathcal K_r
\]
which collapses $W$ to the vertex $z_r$, and let $\eta_*\co W\to B\mathbf 1$ be the unique map. Then $\eta_r|_W=B\Fc(q_*\to q_r)\eta_*$, so these maps define an objectwise weak homotopy equivalence of strict diagrams
\[
\Dc_{\Uc}\Longrightarrow B\Fc.
\]
Theorem~\ref{thm:categorical-sancho} therefore yields
\[
B\left(\int_{Q_{\Uc}^{\op}}\Fc\right) \simeq_w \bigvee_{r=1}^{m}S^1_r.
\]
The full intersection poset $P_{\Uc}$ has $2^m-1$ objects, whereas $Q_{\Uc}$ has only $m+1$, since it discards the intersection patterns which are not realized as exact membership sets. This compression concerns the indexing poset. The Grothendieck construction above has $2m+1$ objects, while the bouquet also admits the more economical model obtained by gluing $m$ Kronecker categories along a common object, which has $m+1$ objects.
\end{example}
\begin{example}
\label{ex:membership-mobius-bands}
Let $m\geq3$. Let $P$ be the crown poset and let $\mathcal K_r$ be a copy of the Kronecker category for each $1\leq r\leq m$. As in Section~\ref{subsec:projective-plane-model}, choose functors $d_r\co P\to\mathcal K_r$ such that $Bd_r\co BP\to B\mathcal K_r$ is a twofold covering of circles.

Let $M_r$ be the mapping cylinder of $Bd_r$, with its copy of $BP$ as boundary. Thus, $M_r$ is a Möbius band \cite[Example~1.35]{Hatcher2002}. Form $X_m$ by identifying the boundaries of $M_1,\ldots,M_m$ with one another. When $m=2$, this construction gives the Klein bottle; we assume $m\geq3$ because, for $m=2$, both indexing posets considered below have three objects and there is no compression.

Choose a small open collar $W$ of the common boundary in $X_m$, and put $U_r=W\cup\operatorname{int}(M_r)$. Then $\Uc=\{U_1,\ldots,U_m\}$ is an open cover, and every intersection involving at least two distinct members is $W$. Moreover, $W\simeq BP$ and $U_r\simeq B\mathcal K_r$.

The only exact membership patterns are $\{r\}$, for $1\leq r\leq m$, and $\{1,\ldots,m\}$. Hence $Q_{\Uc}$ has objects $q_1,\ldots,q_m,q_*$, with $q_r<q_*$. Define $\Fc\co Q_{\Uc}^{\op}\to\cat$ by $\Fc(q_*)=P$, $\Fc(q_r)=\mathcal K_r$, and $\Fc(q_*\to q_r)=d_r$.

These models are compatible as a strict diagram. Projection along the collar yields a homotopy equivalence $\eta_*\co W\to BP$. For each $r$, define $\eta_r\co U_r\to B\mathcal K_r$ by the mapping-cylinder retraction on $M_r$ and, on $W\cap M_s$ for $s\ne r$, by $Bd_r\circ\eta_*$. The definitions agree along the common boundary, and $\eta_r$ is a homotopy equivalence. In particular, $\eta_r|_W=Bd_r\circ\eta_*$, so the maps $\eta_*$ and $\eta_r$ define an objectwise homotopy equivalence of strict diagrams $\Dc_{\Uc}\Rightarrow B\Fc$. Theorem~\ref{thm:categorical-sancho} therefore gives
\[
B\left(\int_{Q_{\Uc}^{\op}}\Fc\right)\simeq_w X_m.
\]
The resulting category is finite and acyclic, with $2m+4$ objects.
Every nonempty formal intersection contains $W$, so $P_{\Uc}$ has $2^m-1$ objects, whereas $Q_{\Uc}$ has only $m+1$. Unlike a good-cover model, the membership poset alone does not capture $X_m$: the functors $d_r$ retain the degree-two boundary maps. In particular, van Kampen's theorem gives
\[
\pi_1(X_m)\cong \langle a_1,\ldots,a_m \mid a_1^2=a_2^2=\cdots=a_m^2\rangle.
\]
\end{example}
\appendix
\section{Homotopy invariance of the simplicial replacement}
\label{app:homotopy-invariance}
We justify the homotopy invariance property used in Theorem~\ref{thm:hocolim-invariance}. The usual elementary statement for the Bousfield--Kan construction assumes that the objects of the diagrams are cofibrant; see \cite[Proposition~4.7]{DuggerHocolimPrimer}. Since the diagrams considered in this article contain open subsets which need not be cofibrant, we use the corresponding result for simplicial spaces with free degeneracies. We follow \cite{DuggerIsaksen2004}.
Recall that a simplicial space $Z$ has \emph{free degeneracies} if there are subspaces $N_nZ\subseteq Z_n$ such that
$$
Z_n \cong \coprod_{\theta\colon[n]\twoheadrightarrow[k]}N_kZ,
$$
where the coproduct runs over the order-preserving surjections and the summand indexed by $\theta$ is mapped into $Z_n$ by the corresponding degeneracy operator.
\begin{lemma}
\label{lem:srep-free-degeneracies}
For every diagram $X\co\I\to\Top$, its simplicial replacement $\operatorname{srep}(X)$ has free degeneracies.
\end{lemma}
\begin{proof}
An $n$-simplex of the indexing nerve $\N\I$ is a string $i_0\to\cdots\to i_n$. Such a string is degenerate precisely when one of its morphisms is an identity. Every string is uniquely obtained from a string containing no identity morphisms by inserting identities. Since the degeneracy operators of $\operatorname{srep}(X)$ insert identity morphisms in the indexing strings, the required decomposition follows.
\end{proof}
\begin{prop}
\label{prop:hocolim-invariance-appendix}
Let $X,Y\co\I\to\Top$ be diagrams and let $\alpha\co X\Rightarrow Y$ be an objectwise weak homotopy equivalence. Then the induced map
$$
\left|\operatorname{srep}(X)\right|\longrightarrow \left|\operatorname{srep}(Y)\right|
$$
is a weak homotopy equivalence.
\end{prop}
\begin{proof}
In simplicial degree $n$, the induced map is
$$
\coprod_{i_0\to\cdots\to i_n}X(i_0) \longrightarrow \coprod_{i_0\to\cdots\to i_n}Y(i_0),
$$
and its restriction to the summand indexed by $i_0\to\cdots\to i_n$ is $\alpha_{i_0}$. It is therefore a weak homotopy equivalence in every simplicial degree.
By Lemma~\ref{lem:srep-free-degeneracies}, both simplicial replacements have free degeneracies. The realization theorem for simplicial spaces with free degeneracies now implies that a levelwise weak homotopy equivalence between them induces a weak homotopy equivalence on realizations; see \cite[Theorem~1.2 and Appendix~A]{DuggerIsaksen2004}. Hence
$$
\left|\operatorname{srep}(X)\right| \longrightarrow \left|\operatorname{srep}(Y)\right|
$$
is a weak homotopy equivalence.
\end{proof}
\bibliographystyle{plain}
\bibliography{biblio}
\end{document}